\documentclass{ws-ijgmmp}
\usepackage{amsmath,eucal,mathbbol,eufrak,amscd}  
\usepackage[bookmarks=false]{hyperref}

\title{\uppercase{H-twisted Courant algebroids}}
\author{\uppercase{Melchior Gr{\"u}tzmann}}
\address{Department of Mathematics, Northwestern Polytechnical University,\\
  Chang'an Campus, Xi'an 710129, Peoples Republic of China,\\
  \email{melchiorG@gMail.com}  }
\date{June 2012}

\newcommand\defbb[2]{\def#1{{\mathbb{#2}}}}

\newcommand\defcal[2]{\def#1{{\mathcal{#2}}}}
\newcommand\deffrak[2]{\def#1{{\mathfrak{#2}}}}
\newcommand\defrm[2]{\def#1{{\mathrm{#2}}}}

\def\:{\colon}
\def\[{\begin{equation}}
\def\]{\end{equation}}
\def\<{\langle}
\def\>{\rangle}
\def\conn_#1{\nabla_{\!\!#1\,}}
\DeclareMathOperator\ad{ad}
\def\cycl{\text{cycl.}}
\defrm\ud{d}
\defrm\uD{D}
\defcal\D{D}
\DeclareMathOperator\Der{Der}

\let\epsilon=\varepsilon

\defcal\E{E}
\newcommand\Esforms[2][\bullet]{\Omega^{#1}_M(E,#2)}
\deffrak\g{g}

\DeclareMathOperator\im{im}
\def\ins{i}
\defbb\k{k}
\defcal\Lie{L}
\defbb\NN{N}
\def\tri{\triangleright}

\def\smooth{{\mathrm{C}^\infty}}

\let\subset=\subseteq

\def\unsh{{\text{unshuffles}}}
\defbb\RR{R}
\defbb\ZZ{Z}

\providecommand{\url}[1]{{\def~{{\textasciitilde}}\texttt{#1}}}
\providecommand\href[2]{\rulebox{\texttt{#2}}}
\providecommand\doi[1]{\href{http://dx.doi.org/#1}{DOI: #1}}
\providecommand{\eprint}[2][]{\href{http://arXiv.org/abs/#2}{#2}}  
\providecommand\Eprint[2][]{\href{http://arXiv.org/abs/#2}{arXiv:#2}}  

\begin{document}
\maketitle
\received{June 2012}
\revised{July 1901}
\bigskip
\begin{abstract}
We generalize Hansen--Strobl's definition of Courant algebroids twisted by a 4-form on the base manifold such that the twist $H$ of the Jacobi identity is a four-form in the kernel of the anchor map and is closed under a naturally occurring exterior covariant derivative.  We give examples and define a cohomology.
\end{abstract}

\keywords{twist of the Courant bracket; Courant algebroid; cohomology of algebroids.}

\section{Introduction}\label{s:intro}
Courant algebroids were introduced by Liu, Weinstein, and Xu in \cite{Xu97} in order to describe the double of a Lie bialgebroid.  They were further investigated by Roytenberg beginning during his Ph.D. studies and a formulation in terms of a Dorfman bracket was discovered \cite{Royt99} as well as the fitting into a two-term $L_\infty$-algebra \cite{Royt98}.  In \cite{Str09} Hansen and Strobl discovered four-form twisted Courant algebroids arising naturally in the Courant sigma model with a Wess--Zumino boundary term.  These $H$-twisted Courant algebroids were further investigated by Liu and Sheng in \cite{SL10b} where the observation was made that exact $H$-twisted Courant algebroids, they fit into a short exact sequence with the tangent and cotangent bundle, always have an exact four-form $H$.  In this paper we want to generalize the notion of $H$-twist and exhibit examples that do not come from an exact or even closed four-form.  The idea is analogous to $H$-twisted Lie algebroids (introduced in \cite{Gru10}) that guided from an exterior covariant derivative (Proposition~\ref{p:D}) that occurs naturally for strongly anchored almost Courant algebroids with anchor $\rho$ on the exterior algebra of sections of $\ker\rho$, one permits the Jacobiator to be a $\ker\rho$-four-form closed under the exterior covariant derivative.  We will give examples of generalized exact four-forms, \textit{i.e.}\ starting from a Courant algebroid with anchor $\rho$ and a $\ker\rho$-three-form with a certain integrability condition we define a Dorfman bracket together with a (nontrivial) $\ker\rho$-four-form $H$ that fit under the above idea.

Since already the definition of the closed generalized four-form requires sections of a possibly singular vector bundle, we also give a definition generalizing Roytenberg's idea of Courant--Dorfman algebras in \cite{Royt09}.

Furthermore, we carry over the idea of Sti{\'e}non and Xu \cite{SX07} to define cochains as a subset of the exterior algebra of the $H$-twisted Courant algebroid such that the naive expression of a differential by the formula that holds for Lie algebroids actually gives a cochain again and squares to 0 in Theorem~\ref{thm:naive}.  We end the treatment with the obvious generalization of Dirac structures to $H$-twisted Courant algebroids and Strobl's as well as Sheng--Liu's idea \cite{SL10b} that such Dirac structures give $H$-twisted Lie algebroids.

In the mean-time parallel developments have shown that it is possible to simplify the definition of H-twisted Courant algebroids, see \cite{LSX12a}.

\medskip
The paper is organized as follows.  In Section~\ref{s:prelim} we give a short summary of the definition of Courant algebroid, two-term $L_\infty$-algebra introduced by Baez and Crans \cite{Baez:03vi} and Roytenberg--Weinstein's observation that together with the skew-symmetric bracket the Courant algebroid gives such a two-term $L_\infty$-algebra.  In Subsection~\ref{s:D} we begin with a definition of strongly anchored almost Courant algebroids and their natural  covariant derivative on the kernel of the anchor map.  We continue with the definition of $H$-twisted Courant algebroids and some examples.  This part ends with the definition of an $H$-twisted Courant--Dorfman algebra.  In Section~\ref{s:naive} we define the naive cohomology of $H$-twisted Courant algebroids.  In the last section we generalize the notion of Dirac structures and give examples of $H$-twisted Lie algebroids.

\section{Preliminaries}\label{s:prelim}
Remember the definition of Courant algebroid.  This goes back to Liu--Weinstein--Xu in \cite{Xu97}.  We take the version of Roytenberg in \cite[2.6]{Royt99}.
\begin{definition}  A Courant algebroid is a vector bundle $E\to M$ together with an $\RR$-bilinear (non-skewsymmetric) bracket $[.,.]\:\Gamma(E)\otimes\Gamma(E)\to\Gamma(E)$, a morphism of vector bundles $\rho\:E\to TM$, and a symmetric non-degenerate bilinearform $\<.,.\>\:E\otimes E\to \RR\times M$ subject to the following axioms
\begin{align}
   [\phi,[\psi_1,\psi_2]] &=[[\phi,\psi_1],\psi_2] +[\psi_1,[\phi,\psi_2]], \\ 
  [\phi,f\cdot\psi] &= \rho(\phi)[f]\cdot\psi +f\cdot[\phi,\psi], \\ 
  [\psi,\psi] &= \tfrac12\rho^*\ud\<\psi,\psi\>, \\ 
  \rho(\phi)\<\psi,\psi\> &= 2\<[\phi,\psi],\psi\>. 
\end{align}  where $\phi,\psi_i\in\Gamma(E)$, $f\in\smooth(M)$, and $\ud$ is the de Rham differential of the smooth manifold $M$.
\end{definition}
In what follows we will identify $E^*$ with $E$ via the symmetric non-degenerate bilinearform $\<.,.\>$.

From \cite{Baez:03vi} we take the following definition of a two-term $L_\infty$-algebra.
\begin{definition}  A two-term $L_\infty$-algebra is a two-term complex $0\to V_1\xrightarrow{\partial} V_0\to 0$ together with three more maps
\begin{align}
  [.,.]\:V_0\wedge V_0&\to V_0, \nonumber\\
  \tri\:V_0\otimes V_1&\to V_1, \nonumber\\
  l_3\: V_0\wedge V_0\wedge V_0 &\to V_1  \nonumber
\intertext{Subject to the rules}
  [\phi,\partial f] &= \partial(\phi\tri f)  \label{n=2}\\
  (\partial f)\tri g + (\partial g)\tri f &= 0  \label{n=2b}\\
  [\phi_1,[\phi_2,\phi_3]\,] +\cycl &= \partial l_3(\phi_1,\phi_2,\phi_3)  \label{n=3}\\
  \phi_1\tri(\phi_2\tri f) -\phi_2\tri(\phi_1\tri f) -[\phi_1,\phi_2]\tri f &= l_3(\phi_1,\phi_2,\partial f)  \label{n=3b}\\
  l_3([\phi_1,\phi_2]\wedge\phi_3\wedge\phi_4) +\phi_1\tri l_3(\phi_2&\wedge\phi_3\wedge\phi_4) +\unsh = 0   \label{n=4}
\end{align} where $\phi_i\in V_0$ and $f\in V_1$.
\end{definition}

As Roytenberg--Weinstein observed, the Courant algebroid gives rise to a two-term $L_\infty$-algebra with the identifications $V_0=\Gamma(E)$, $V_1=\smooth(M)$, $\partial=l_1=\rho^*\circ\ud$, $l_2(\psi_1,\psi_2)=[\psi_1,\psi_2]-\tfrac12\rho^*\ud\<\psi_1,\psi_2\>$, $\psi\tri f=\tfrac12\<\psi,\partial f\>$, and $l_3(\psi_1,\psi_2,\psi_3)=\tfrac16\<[\psi_1,\psi_2],\psi_3\>+\cycl$.

Since in the treatment of $H$-twisted Courant algebroids we will encounter sections of possibly singular vector bundles, we will also introduce the notion of Lie--Rinehart \cite{Rin63} as well as Courant--Dorfman algebras \cite{Royt09}.  For this purpose let $\k$ be a commutative ring (with unit 1) and $R$ a commutative $\k$-algebra.
\begin{definition} A Lie--Rinehart algebra $(R,\E,[.,.],\rho)$ is an $R$-module $\E$ together with a $\k$-Lie algebra structure $[.,.]$ on $\E$ and an $R$-linear representation $\rho\:E\to\Der(R)$ subject to the rules
\begin{align*}
  0&=[\psi_1,[\psi_2,\psi_3]] +\cycl,  \\
  [\psi,f\cdot\phi] &= \rho(\psi)[f]\cdot\phi +f\cdot[\psi,\phi],  \\
  \rho[\phi,\psi] &=[\rho(\phi),\rho(\psi)]_{\Der(R)}.
\end{align*}
\end{definition}

Examples are $\E$ the sections of a Lie algebroid $E\to M$ with $R=\smooth(M)$.

\begin{definition}  Let $\k$ contain $\tfrac12$.  A Courant--Dorfman algebra $(R,\E,\<.,.\>,\rho,[.,.])$ consists of an $R$-module $\E$, a symmetric $R$-bilinearform $\<.,.\>\:\E\otimes_R\E\to R$, a derivation $\partial\:R\to\E$, and a $\k$-bilinear (non-skewsymmetric) bracket $[.,.]\:\E\otimes\E\to\E$ subject to the rules
\begin{align*}
  [\psi,f\cdot\phi] &= \rho(\psi)[f]\cdot\phi +f\cdot[\psi,\phi],  \\
  \<\psi,\partial\<\phi,\phi\>\> &= 2\<[\psi,\phi],\phi\>,  \\
  [\psi,\psi] &= \tfrac12\partial\<\psi,\psi\>, \\
  [\phi,[\psi_1,\psi_2]] &= [[\phi,\psi_1],\psi_2] +[\psi_1,[\phi,\psi_2]],  \\
  [\partial f,\phi] &= 0,  \\
  \<\partial f,\partial g\> &=0
\end{align*} for all $\phi,\psi_i\in\E$, $f,g\in R$.  We call it almost Courant--Dorfman algebra iff only the first three rules hold.
\end{definition}

Examples are $\E$ the sections of a Courant algebroid $E\to M$, $R=\smooth(M)$, $\partial=\rho^*\circ\ud$; but also Lie--Rinehart algebras with trivial pairing $\<.,.\>\equiv0$.

\section{$H$-twisted Courant algebroids}
\subsection{Covariant derivative for strongly anchored almost Courant algebroids}\label{s:D}
\begin{definition}  A \emph{strongly anchored almost Courant algebroid} is a vector bundle $E\to M$ together with a bilinear (non-skewsymmetric) bracket $[.,.]\:\Gamma(E)\otimes\Gamma(E)\to\Gamma(E)$, a symmetric nondegenerate bilinearform $\<.,.\>\:E\otimes E\to \RR\times M$, and a vector bundle morphism $\rho\:E\to TM$, called the anchor subject to the axioms
\begin{align}
  \rho[\phi,\psi] &= [\rho(\phi),\rho(\psi)]_{TM}, \label{rhoMor}\\
  [\phi,f\cdot\psi] &= \rho(\phi)[f]\cdot\psi +f\cdot[\phi,\psi], \label{Leibniz1} \\
  [\psi,\psi] &= \tfrac12\rho^*\ud\<\psi,\psi\>, \label{nSkew1}\\
  \rho(\phi)\<\psi,\psi\> &= 2\<[\phi,\psi],\psi\>. \label{adInv1}
\end{align}
\end{definition}

Given a smooth anchor map $\rho\:E\to TM$ we define the $\Omega^\bullet_M(\ker\rho)$ to be the smooth sections $\Gamma(\wedge^\bullet E)$ that lie in the kernel of $\tilde\rho:\wedge^\bullet E\to TM\otimes \wedge^{\bullet-1}E:\psi_1\wedge\psi_2\mapsto \rho(\psi_1)\otimes\psi_2-\rho(\psi_2)\otimes\psi_1$ and extended correspondingly for more terms.

Following an idea of Sti{\'e}non and Xu \cite{SX07} we define an exterior covariant derivative on these cochains by the formula that holds for Lie algebroids.

\begin{proposition}\label{p:D}  The following is an exterior covariant derivative, \textit{i.e.}\ $\smooth(M)$-linear in the occuring $\psi_i\in\Gamma(M)$.  For $\alpha\in\Omega^p_M(\ker\rho)$ define
\[\begin{split}  \<\D\alpha,\psi_0\wedge\dots\psi_p\> =& \sum_{i=0}^p (-1)^i\rho(\psi_i)\<\alpha,\psi_0\wedge\dots\hat\psi_i\dots\psi_p\> \\
  &+\sum_{i<j} (-1)^{i+j}\<\alpha,[\psi_i,\psi_j]\wedge\psi_0\dots\hat\psi_i\dots\hat\psi_j\dots\psi_p\>
\end{split}  \label{D}
\]  $\D$ maps $\Omega^p(\ker\rho)\to\Omega^{p+1}(\ker\rho)$ and fulfills the Leibniz rule
\[  \D(\alpha\wedge\beta) = (\D\alpha)\wedge\beta +(-1)^{|\alpha|}\alpha\wedge\D\beta \;.  \label{Dleibn}
\]
\end{proposition}
\begin{proof}  The main difference to Lie algebroids is that the bracket is not skewsymmetric.  However the non-skewsymmetric part of the bracket vanishes when inserted into $\alpha$.  The rest is now a straightforward calculation.  For the last statement note that $\D$ is a first order odd differential operator.
\end{proof}

Note that it is also possible to split a $\ker\rho$-$p+k$-form $\alpha$ as a $\ker\rho$-$p$-form with values in the $k$-fold exterior power of $\ker\rho$.  We will denote any possible splitting as $\tilde\alpha$.

\subsection{Definition and examples}\label{s:defn}
\begin{definition}  An $H$-twisted Courant algebroid is a vector bundle $E\to M$ together with an $\RR$-bilinear (non-skewsymmetric) bracket $[.,.]\:\Gamma(E)\otimes\Gamma(E)\to\Gamma(E)$, a morphism of vector bundles $\rho\:E\to TM$, a symmetric non-degenerate bilinearform $\<.,.\>\:E\otimes E\to \RR\times M$, and a $\ker\rho$-four-form $H\in\Omega^4_M(\ker\rho)$ subject to the following axioms
\begin{align}
   \tilde{H}(\phi,\psi_1,\psi_2) &= [\phi,[\psi_1,\psi_2]]-[[\phi,\psi_1],\psi_2]-[\psi_1,[\phi,\psi_2]], \label{HJacobi} \\
  \D H &= 0,  \label{DH}\\
  [\phi,f\cdot\psi] &= \rho(\phi)[f]\cdot\psi +f\cdot[\phi,\psi], \label{Leibniz} \\
  [\psi,\psi] &= \tfrac12\D\<\psi,\psi\>, \label{nSkew}\\
  \rho(\phi)\<\psi,\psi\> &= 2\<[\phi,\psi],\psi\>. \label{adInv}
\end{align}  where $\phi,\psi_i\in\Gamma(E)$, $f\in\smooth(M)$, and $\D$ is the covariant derivative defined in the previous subsection.
\end{definition}

\begin{lemma}  $\rho$ is a morphism of brackets, \textit{i.e.}
\[  \rho[\phi,\psi] = [\rho(\phi),\rho(\psi)] \;.
\]
\end{lemma}
\begin{proof}  Start from $[\rho(\phi),\rho(\psi)][f]\cdot\chi]$ for $\phi,\psi,\chi\in\Gamma(E)$, $f\in\smooth(M)$ and expand using the Leibniz rule to iterated brackets. Then use the Jacobi identity \eqref{HJacobi}, and note that the H-contributions cancel, because $H$ is $\smooth(M)$-linear.
\end{proof}

\begin{example}\begin{enumerate}\addtocounter{enumi}{-1}\item  Courant algebroids are exactly the $H$-twisted Courant algebroids where $H=0$.
\item  Analogously to the $H$-twisted Lie algebroids we start with an untwisted Courant algebroid $(E,\<.,.\>,\rho,[.,.]_0)$ and make the general ansatz 
 \[ [\phi,\psi]_B:= [\phi,\psi]_0+\tilde{B}(\phi,\psi)
 \] where $B\in\Omega^3_M(\ker\rho)$.  The Jacobiator of this bracket is
 \[ \tilde{H}:= \widetilde{\D_0 B} +\tilde{B}^2 \]
 where $\tilde{B}^2(\psi_1,\psi_2,\psi_3):=\tilde{B}(\tilde{B}(\psi_1,\psi_2),\psi_3)+\cycl$ and the condition $\D H=0$ reads as
 \[ 0=\D_B H=\D_0\tilde{B}^2+\tilde{B}\D_0B+\tilde{B}^3\;. \]  In the computation we use the fact observed by Sti{\'e}non--Xu that the naive differential $\D_0$ squares to 0.  If we start with a Courant algebroid with $\ker\rho$ of rank at most 4, then every $B\in\Omega^3_M(\ker\rho)$ gives a twisted Courant algebroid.
 
In general, if we can find nontrivial solutions of this nonlinear first order PDE, we can provide nontrivial examples of $H$-twisted Courant algebroids.

\item  One particular case arises when we start with a Courant algebroid $(E,\rho,[.,.],h)$ twisted by a closed 4-form $h\in\Omega^4(M)$ in the sense of Hansen--Strobl \cite{Str09}.  If we pull it back to $\Omega^4_M(\ker\rho)$ via $\rho^*$ we obtain an H-twisted Courant algebroid, because $\im\rho^*\subset\ker\rho$ as well as
\begin{lemma}
\[  \D\circ\rho^* = \rho^*\circ\ud
\]
\end{lemma}  which follows from the morphism property of the anchor map.

\item  Given an H-twisted Lie algebra (an almost Lie algebra $\g$ whose Jacobi identity is twisted by a three-form with values in $\g$ and $\D\g=0$ for the corresponding $\D$), then this augments to an $H$-twisted Courant algebroid over a point iff we can find an $\ad$-invariant symmetric bilinearform $\<.,.\>$ for it and $H$ is then skew-symmetric.
\end{enumerate}
\end{example}

\begin{proposition}  The $H$-twisted Courant algebroid $(E,\rho,[.,.],H)$ is a two-term $L_\infty$-algebra with the identifications $V_0:=\Gamma(E)$, $V_1:=\Gamma(\ker\rho)$, and the operations
\begin{align}
  \partial=l_1\:V_1 &\subset V_0, \\
  l_2\:V_0\wedge V_\bullet&\to V_\bullet: (\psi_1,\psi_2)\mapsto [\psi_1,\psi_2]-\tfrac12\D\<\psi_1,\psi_2\>, \\
  l_3\:\wedge^3V_0 &\to V_1: (\psi_1,\psi_2,\psi_3)\mapsto H(\psi_1,\psi_2,\psi_3)+\tfrac16\D\<[\psi_1,\psi_2],\psi_3\>+\cycl
\end{align}
\end{proposition}
The correction in the bracket $l_2$ and in the Jacobiator $l_3$ are analogous to Roytenberg \cite{Royt99} and therefore fit the Courant case.  
\begin{proof}  Straightforward but lengthy calculation.
\end{proof}

\subsection{$H$-twisted Courant--Dorfman algebras}
Let $\k$ be a commutative ring (with unit 1) that contains $\tfrac12$.  Analogously to Roytenberg \cite{Royt09} we define a strongly anchored almost Courant--Dorfman algebra as:
\begin{definition}  A strongly anchored almost Courant--Dorfman algebra $(R,\E,(.,.\>,\D_0,[.,.])$ is an $R$-module $\E$ together with a symmetric $R$-bilinearform $\<.,.\>\:\E\otimes_R\E\to R$ such that $\kappa\:\E\to\E^*:\psi\mapsto\<\phi,.\>$ is an isomorphism of $R$-modules, a derivation $\D_0\:R\to\E$, and a $\k$-bilinear (non-skewsymmetric) bracket $[.,.]:\E\otimes\E\to\E$ subject to the rules
\begin{align}
  [\psi,f\cdot\phi] &= \<\psi,\D_0 f\>\cdot\phi +f\cdot[\psi,\phi],  \\
  \<\psi,\D_0\<\phi,\phi\>\> &= 2\<[\psi,\phi],\phi\>,  \\
  [\phi,\phi] &= \tfrac12\D_0\<\phi,\phi\>,  \\
  \<[\psi,\phi],\D_0f\> &= \<\phi,\D_0\<\psi,\D_0f\>\> -\<\psi,\D_0\<\phi,\D_0f\>\>
\end{align}
\end{definition}
Examples are $\E$ the sections of a strongly anchored almost Courant algebroid $(E,\<.,.\>,\rho,[.,.])$.

These strongly anchored almost Courant--Dorfman algebras inherit a derivative of degree 1 on the exterior algebra $C^p(\E,\D_0):=\E^{\wedge p}\cap\ker\ins_{\D_0 R}$ as before:
\[\begin{split}  \<\D\alpha,\psi_0\wedge\dots\psi_p\> :=&\sum_{i=0}^p (-1)^i\<\psi_i,\D_0\<\alpha,\psi_0\wedge\dots\hat\psi_i\dots\psi_n\>\> \\
  &+\sum_{i<j} (-1)^{i+j}\<\alpha,[\psi_i,\psi_j]\wedge\psi_0\dots\hat\psi_i\dots\hat\psi_j\dots\psi_p\>
\end{split}\]  Note that in particular $(\D|R)=\D_0$.

Therefore we can define  $H$-twisted Courant--Dorfman algebras analogously to Roytenberg's definition.
\begin{definition}  An $H$-twisted Courant--Dorfman algebra $(R,\E,\<.,.\>,\D_0,[.,.],H)$ is an $R$-module $\E$ together with a symmetric $R$-bilinearform $\<.,.\>\:\E\otimes_R\E\to R$ such that $\kappa\:\E\otimes_R\E\to R:\psi\mapsto\<\psi,.\>$ is an isomorphism of $R$-modules, a derivative $\D_0\:R\to\E$, a $\k$-bilinear (non-skewsymmetric) bracket $[.,.]\:\E\otimes\E\to\E$, and a $C^4(E,\D_0)$-form $H$ subject to the rules
\begin{align}
  [\psi,f\cdot\phi] &= \<\psi,\D_0 f\>\cdot\phi +f\cdot[\psi,\phi],  \\
  \<\psi,\D_0\<\phi,\phi\>\> &= 2\<[\psi,\phi],\phi\>,  \\
  [\phi,\phi] &= \tfrac12\<\phi,\phi\>,  \\
  \tilde{H}(\phi,\psi_1,\psi_2) &= [\phi,[\psi_1,\psi_2]] -[[\phi,\psi_1],\psi_2] -[\psi_1,[\phi,\psi_2]], \\
  \D H &= 0, \\
  [\D_0f,\phi] &= 0,\\
  \<\D_0f,\D_0g\> &= 0
\end{align}  where $\phi,\psi_i\in\E$, $f,g\in R$ and $\D$ the extension of $\D_0$ as defined above.
\end{definition}
Examples are $\E$ the sections of an H-twisted Courant agebroid $(E,\<.,.\>,\rho,[.,.],H)$.

\section{Naive Cohomology}\label{s:naive}
\begin{proposition}  The covariant derivative $\D$ of Subsection~\ref{s:D} does not square to 0 in general, instead it fulfills for $H$-twisted Courant algebroids
\begin{align}
  \<\D^2f,\psi_0\wedge\psi_1\> &= 0, \\
  \<\D^2\phi, \psi_0\wedge\psi_1\> &= H(\phi,\psi_0,\psi_1), \\
  \D^2(\alpha\wedge\beta) &= (\D^2\alpha)\wedge\beta+ \alpha\wedge\D^2\beta
\end{align}  for $f\in\smooth(M)$, $\phi\in\Gamma(\ker\rho)$, $\alpha,\beta\in\Omega^\bullet_M(\ker\rho)$, and $\psi_i\in\Gamma(E)$.
\end{proposition}
\begin{proof}  The proof is analogous to the one for $H$-twisted Lie algebroids, namely the first statement follows from the morphism property of $\rho$, the second statement is a reformulation of the Leibniz rule, and the last statement follows from the graded Leibniz rule \eqref{Dleibn}.
\end{proof}

\begin{theorem}[Naive cohomology]\label{thm:naive}  The cochains
\begin{align}  C^p(E,\rho,H) &:= \Omega^p(\ker\rho)\cap \ker\tilde{H}
\intertext{together with the derivative}
  d\:C^p(E,\rho,H) &\to C^{p+1}(E,\rho,H) : \alpha \mapsto \D\alpha
\end{align}  form a cochain complex.
\end{theorem}
\begin{proof}  It remains to check that $\D$ maps $\tilde{H}$-closed forms to $\tilde{H}$-closed forms.  This follows from the property
 \[ [\D,\tilde{H}]= \widetilde{\D H} = 0
 \] due to the axiom \eqref{DH}.
\end{proof}

The corresponding notion of naive cochains for Courant--Dorfman algebras is
\[  C^p(\E,\D_0,H):= \ker \tilde{H}|\E^{\wedge p}\cap\ker\ins_{\D_0R}.
\]

\section{Dirac Structures and $H$-twisted Lie Algebroids}\label{s:Dirac}
Given an $H$-twisted Courant algebroid (with bilinearform) of split signature, we define a Dirac structure in the usual way.
\begin{definition}  Given an $H$-twisted Courant algebroid $(E,\<.,.\>,[.,.],\rho,H)$, we define
\begin{enumerate}\item  an \emph{isotropic subbundle} $L\subset E$ as a vector subbundle over $M$ such that $\<L,L\>\equiv0$.  If the bilinearform is of split signature, we can consider maximal isotropic subbundles with respect to inclusion and call them Lagrangean subbundles.
\item an \emph{integrable subbundle} $L\subset E$ when the bracket closes on the sections of $L$, \textit{i.e.}\ $[\Gamma(L),\Gamma(L)]\subset\Gamma(L)$.
\item a \emph{Dirac structure} as a maximal isotropic integrable subbundle in an $H$-twisted Courant algebroid of split signature.
\end{enumerate}
\end{definition}

Compare this with the definition of $H$-twisted Lie algebroids (taken from \cite{Gru10}):
\begin{definition}\label{d:HLie}  An H-twisted Lie algebroid is a vector bundle $E\to M$ together with a bundle map $\rho\:E\to TM$ (called the anchor), a section $H\in\Esforms[3]{\ker\rho}$, and a skew-symmetric bracket $[.,.]\:\Gamma(E)\wedge\Gamma(E)\to\Gamma(E)$ subject to the axioms
\begin{align}
 [\phi,[\psi_1,\psi_2]] &= [[\phi,\psi_1],\psi_2] +[\psi_1,[\phi,\psi_2]] +H(\phi,\psi_1,\psi_2)  \label{Jacobi} \\
 [\phi, f\cdot\psi] &= \rho(\phi)[f]\cdot\psi +f\cdot[\phi,\psi] \label{LLeibniz} \\
 \uD H &= 0 \label{Hclosed}
\end{align}  where $f\in\smooth(M)$, $\phi,\psi,\psi_i\in\Gamma(E)$ and $\uD$ is the one defined for anchored almost Lie algebroids analogous to \eqref{D}, but $\rho$ replaced by
 \[ \conn_\psi v :=[\psi,v]
 \] for every $\psi\in\Gamma(E)$ and $v\in\Gamma(\ker\rho)$ which is an $E$-connection on $\ker\rho$.
\end{definition}

We have the immediate consequence.
\begin{proposition}  Given an $H$-twisted Courant algebroid $(E,H)$ of split signature.  Then every Dirac structure $L\subset E$ is an $H$-twisted Lie algebroid.  In particular the twist $\tilde{H}$ induces a $\uD$-closed $L$-three-form with values in $\ker\rho|L$.
\end{proposition}

\section*{Acknowledgments}
Research on this paper was conducted during the stay at Sun Yat-sen University and partially supported by NSFC(10631050 and 10825105) and NKBRPC(2006CB805905).  The paper was finished at Northwestern Polytechnical University.  I am grateful to Z.-J. Liu for comments on an earlier version of this paper.


\end{document}